\documentclass[reqno]{amsart}
\usepackage{amssymb, amsmath}
\newtheorem{theorem}{Theorem}
\newtheorem{acknowledgement}[theorem]{Acknowledgement}

\newtheorem{corollary}[theorem]{Corollary}

\newtheorem{lemma}[theorem]{Lemma}

\begin{document}

\title[The Solvability of Inverse Problem]
{Inverse Scattering Problem for a Piecewise Continuous Sturm - Liouville Equation with Eigenparameter Dependence in the Boundary Conditon}
\author[Kh. R. Mamedov, Nida P. Kosar and F. Ayca Cetinkaya]{Kh. R. Mamedov$^{1}$, Nida P. Kosar$^{2}$ and F. Ayca Cetinkaya$^{3,*}$} 
\address{$^{1,3}$Science and Letters Faculty, Mathematics Department,
Mersin University, 33343, Turkey}
\address{$^{2}$
 Department of Elementary Mathematics Education, Gaziantep University, 27310, Gaziantep, Turkey}
\thanks{$^{*}$corresponding author}
\email{$^{1}$hanlar@mersin.edu.tr, $^{2}$npkosar@gantep.edu.tr, $^{3}$ faycacetinkaya@mersin.edu.tr}
\date{}
\maketitle
\begin{abstract}
In this study, the inverse problem of the scattering theory on the half line for a piecewise continuous Sturm-Liouville equation with boundary condition depending quadratic on the spectral parameter is considered. The scattering data of the problem is defined, some properties of the
scattering data are investigated. The main equation is derived and uniqueness algorithm to the potential with given scattering data is studied.
\end{abstract}

\section{Introduction}

\quad In this paper, we consider the differential equation 
\begin{equation} \label{1.1}
-y''+q(x)y=\lambda^{2}\rho(x)y \quad 0\leq x<+\infty
\end{equation}
with the boundary condition
\begin{equation} \label{1.2}
\left(\beta_{0}+i\beta_{1}\lambda+\beta_{2}\lambda^{2}\right)y'(0)+\left(\alpha_{0}+i\alpha_{1}\lambda+\alpha_{2}\lambda^{2}\right)y(0)=0
\end{equation}
where $\lambda$ is a spectral parameter, $q(x)$ is a real valued function satisfying the condition 
\[
\int^{+\infty}_{0} (1+x)\left|q(x)\right|dx < + \infty,
\]
$\rho(x)$ is a positive piecewise continuous function
\[
\rho(x)=\left\{
\begin{array}{rl}
\alpha^{2}, & 0\leq x < a,\\
1, & x >a,
\end{array} \right.
\]
as $0< \alpha \neq 1$ and $\alpha_{i}, \beta_{i}$ $(i=0,1,2)$ are real numbers satisfying the conditions 
\begin{equation} \label{1.3}
\delta_{1}:=\alpha_{0}\beta_{1}-\alpha_{1}\beta_{0} \leq 0, \quad \delta_{2}:=\alpha_{0}\beta_{2}-\alpha_{2}\beta_{0} \leq 0, \quad \delta_{3}:=\alpha_{1}\beta_{2}-\alpha_{2}\beta_{1} \geq 0.
\end{equation}

The inverse scattering problem for (\ref{1.1}) is completely solved in \cite{Marchenko} with the boundary condition $y(0)=0$ and in \cite{Levitan,Sargsjan} with $y'(0)-hy(0)=0$ where $h$ is an arbitrary real number. Inverse problem of spectral analysis which has a spectral parameter in the boundary condition is studied in \cite{Fedotova} as regards to spectral function. In \cite{Guseinov} a boundary value problem which consists a second order differential equation with a discontinuous coefficient is studied on the half line. When the coefficient has discontinuity at the point $a$, the solution of the inverse scattering problem was examined as the solution of two inverse problems in the intervals $[0,a]$ and $[a, +\infty)$ in \cite{Gasymov,Darwish}. This discontinuity affects the structure of the representation of the Jost solution and the main equation of the inverse problem. In \cite{Mamedov3,Mamedov1,Mamedov2,Manafov} similar problem is examined as regards to scattering data . This type of boundary conditions arise from applied problems such as the study of heat condition by \cite{Cohen}. In \cite{Huseynov} the inverse scattering problem is dealt on the whole axis. \\

The paper is organized as follows. In Section 2, the scattering data for the boundary value problem (\ref{1.1}), (\ref{1.2}) is defined and properties of the scattering data are examined. In Section 3, the main equation for the inverse problem is constructed and the uniqueness of solution of the inverse problem is given.\\

The function
\[
e_{0}(x,\lambda)=\frac{1}{2}\left(1+\frac{1}{\alpha}\right)e^{i \lambda \mu^{+}(x)}+\frac{1}{2}\left(1-\frac{1}{\alpha}\right)e^{i \lambda \mu^{-}(x)}
\]
is the Jost solution of (\ref{1.1}) when $q(x)\equiv 0$, where $\mu^{\pm}(x)=\pm x\sqrt{\rho(x)}+a(1\mp \sqrt{\rho(x)})$.\\

It is well known from (see \cite{Guseinov,Mamedov1}) that for all $\lambda$ from the closed upper half-plane (\ref{1.1}) has a unique solution $e(x, \lambda)$ which can be represented in the form 
\begin{equation} \label{1.4}
e(x, \lambda)=e_{0}(x, \lambda)+\int_{\mu^{+}(x)}^{+\infty}K(x,t) e^{i \lambda t} dt,
\end{equation} 
where the function $K(x,t) \in L_{1}\left(\mu^{+}(x), +\infty\right)$ satisfies the properties below:
\begin{equation} \label{1.5}
\frac{d}{dx}K(x,\mu^{+}(x))=\frac{1}{4\sqrt{\rho(x)}}\left(1+\frac{1}{\sqrt{\rho(x)}}\right)q(x),
\end{equation}
\begin{equation} \label{1.6}
\frac{d}{dx}\left[K(x,\mu^{-}(x)+0)-K(x,\mu^{-}(x)-0)\right]=\frac{1}{4\sqrt{\rho(x)}}\left(1-\frac{1}{\sqrt{\rho(x)}}\right)q(x),
\end{equation}
if $q(x)$ is differentiable, the kernel $K(x,t)$ satisfies (a.e) the equation
\[
\rho(x)K''_{tt}-K''_{xx}+q(x)K=0, \quad 0 \leq x<+\infty, \quad t>\mu^{+}(x).
\]

Since the function $q(x)$ are real valued and the numbers $\alpha_{i}, \beta_{i}$ $(i=0,1,2)$ are real, the function $\overline{e(x,\lambda)}$ is also a solution of the boundary value problem (\ref{1.1}), (\ref{1.2}) and together with $e(x,\lambda)$ it forms a system of fundamental solutions for real $\lambda \neq 0$, their Wronskian does not depend on $x$ and
\begin{equation} \label{1.7}
W\left\{e(x,\lambda), \overline{e(x, \lambda)}\right\}= e'(x,\lambda)\overline{e(x,\lambda)}-e(x,\lambda)\overline{e'(x,\lambda)}=2i\lambda
\end{equation}
holds.

\section{Scattering Data}

\quad Let us assume that $w(x,\lambda)$ be a solution of equation (\ref{1.1}) satisfying the initial conditions
\[
w(0, \lambda)=\beta_{0}+i\beta_{1}\lambda+\beta_{2}\lambda^{2}, \quad w'(0, \lambda)=-\left(\alpha_{0}+i\alpha_{1}\lambda+\alpha_{2}\lambda^{2}\right)
\]
then the following assertion holds.
\begin{lemma}
For any real number $\lambda \neq 0$ 
\begin{equation} \label{2.1}
\frac{2i \lambda w(x,\lambda)}{E(\lambda)}=\overline{e(x,\lambda)}-S(\lambda)e(x,\lambda)
\end{equation}
is valid, where 
\[
E(\lambda):= \left(\beta_{0}+i\beta_{1}\lambda+\beta_{2}\lambda^{2}\right)e'(0,\lambda)+\left(\alpha_{0}+i\alpha_{1}\lambda+\alpha_{2}\lambda^{2}\right)e(0,\lambda)
\]
and 
\[
S(\lambda):=\frac{\left(\beta_{0}+i\beta_{1}\lambda+\beta_{2}\lambda^{2}\right)\overline{e'(0,\lambda)}+\left(\alpha_{0}+i\alpha_{1}\lambda+\alpha_{2}\lambda^{2}\right)\overline{e(0,\lambda)}}{\left(\beta_{0}+i\beta_{1}\lambda+\beta_{2}\lambda^{2}\right)e'(0,\lambda)+\left(\alpha_{0}+i\alpha_{1}\lambda+\alpha_{2}\lambda^{2}\right)e(0,\lambda)}.
\]

Moreover, the functions $E(\lambda)$ and $S(\lambda)$ have the properties below:
\[
\overline{E(\lambda)}=E(-\lambda), \quad \overline{S(-\lambda)}=S(\lambda).
\]
\end{lemma}
\begin{proof}
\ Since the functions $e(x,\lambda)$ and $\overline{e(x,\lambda)}$ are fundamental solutions of equation (\ref{1.1}) for all real $\lambda \neq 0$ we can write 
\[
w(x,\lambda)=c_{1}(\lambda)e(x,\lambda)+c_{2}(\lambda)\overline{e(x,\lambda)}.
\]
Now let us consider the following relations:
\[
c_{1}(\lambda)e(0,\lambda)+c_{2}(\lambda)\overline{e(0,\lambda)}=\beta_{0}+i\beta_{1}\lambda+\beta_{2}\lambda^{2},
\]
\[
c_{1}(\lambda)e'(0,\lambda)+c_{2}(\lambda)\overline{e'(0,\lambda)}=-\left(\alpha_{0}+i\alpha_{1}\lambda+\alpha_{2}\lambda^{2}\right).
\]
Hence, we have
\[
c_{1}(\lambda)=-\frac{1}{2i \lambda}\left[\left(\beta_{0}+i\beta_{1}\lambda+\beta_{2}\lambda^{2}\right)\overline{e'(0,\lambda)}+\left(\alpha_{0}+i\alpha_{1}\lambda+\alpha_{2}\lambda^{2}\right)\overline{e(0,\lambda)}\right], 
\]
\[
c_{2}(\lambda)=\frac{1}{2i \lambda}\left[\left(\beta_{0}+i\beta_{1}\lambda+\beta_{2}\lambda^{2}\right)e'(0,\lambda)+\left(\alpha_{0}+i\alpha_{1}\lambda+\alpha_{2}\lambda^{2}\right)e(0,\lambda)\right].
\]
Thus, the relation below is valid:
\[
w(x,\lambda)=\frac{-1}{2i\lambda}\left[\left(\beta_{0}+i \beta_{1} \lambda +\beta_{2}\lambda^{2}\right)\overline{e'(0,\lambda)}+\left(\alpha_{0}+i\alpha_{1}\lambda+\alpha_{2}\lambda^{2}\right)\overline{e(0,\lambda)}\right]e(x,\lambda)
\]
\begin{equation} \label{2.2}
+\frac{1}{2i\lambda}\left[\left(\beta_{0}+i\beta_{1}\lambda+\beta_{2}\lambda^{2}\right)e'(0,\lambda)+(\alpha_{0}+i \alpha_{1} \lambda+\alpha_{2}\lambda^{2})e(0,\lambda)\right]\overline{e(x,\lambda).}
\end{equation}
Now let us show that $E(\lambda)\neq 0$ holds for real $\lambda \neq 0$. \\

Assuming $E(\lambda_{0})=0$ as $0 \neq \lambda_{0} \in (- \infty, +\infty)$ we get
\begin{equation} \label{2.3}
e'(0, \lambda_{0})=-\frac{\left(\alpha_{0}+i\alpha_{1}\lambda_{0}+\alpha_{2}\lambda^{2}_{0}\right)}{\left(\beta_{0}+i\beta_{1}\lambda_{0}+\beta_{2}\lambda^{2}_{0}\right)}e(0, \lambda_{0}).
\end{equation} 
If we take into consideration (\ref{1.7}) and (\ref{2.3}) we have
\[
\left|e(0, \lambda_{0})\right|^{2}\left[\frac{\delta_{1}-\lambda_{0}^{2}\delta_{3}}{\left(\beta_{0}+i\beta_{1}\lambda_{0}+\beta_{2}\lambda_{0}^{2}\right)\left(\beta_{0}-i\beta_{1}\lambda_{0}+\beta_{2}\lambda_{0}^{2}\right)}\right]=1
\]
and due to the relations given in (\ref{1.3}), the last equation makes a contradiction to the assumption $E(\lambda_{0})=0$ for all $\lambda_{0}\neq 0$. The validity of (\ref{2.1}) can be easily seen from (\ref{2.2}). Hence, the lemma is proven. 
\end{proof}

The function $S(\lambda)$ which is identified in (\ref{2.1}) is called \textit{the scattering function} of the boundary value problem (\ref{1.1}), (\ref{1.2}). \\

The left hand side of (\ref{2.1}) is clearly a meromorphic function in the upper half plane $Im \lambda >0$ with poles at the zeros of the function $E(\lambda)$.
\begin{lemma}
The function $E(\lambda)$ may have only a finite number of zeros in the half plane $Im \lambda >0$.
\end{lemma}

\begin{proof}
\ Since $E(\lambda) \neq 0$ for all real $\lambda \neq 0$, the point $\lambda=0$ is the possible real zero of the function $E(\lambda)$. Using the analiticity of the function $E(\lambda)$ in upper half plane and the properties of solution (\ref{1.4}) it is obtained that the zeros of the function $E(\lambda)$ form at most countable and bounded set having zero as the only possible limit point. \\

Now, let us show that these zeros are in the half plane $Im \lambda > 0$. Suppose that $\lambda_{1}$ and $\lambda_{2}$ are arbitrary zeros of the function $E(\lambda)$ and consider the following relations
\[
-e''(x,\lambda_{1})+q(x)e(x,\lambda_{1})=\lambda_{1}^{2}\rho(x)e(x,\lambda_{1}),
\]
\[
\overline{-e''(x,\lambda_{2})}+q(x)\overline{e(x,\lambda_{2})}=\overline{\lambda_{1}^{2}}\rho(x)\overline{e(x,\lambda_{2})}.
\]
Multiplying the first equation by $\overline{e(x,\lambda_{2})}$ and the second one by $e(x,\lambda_{1})$, subtracting them side by side and integrating over $\left[0, +\infty\right)$ we have
\begin{equation} \label{2.4}
\left(\lambda_{1}^{2}-\overline{\lambda_{2}}^{2}\right)\int_{0}^{+\infty}e(x,\lambda_{1})\overline{e(x,\lambda_{2})}\rho(x)dx-W\left\{e(0,\lambda_{1}),\overline{e(0,\lambda_{2})}\right\}=0.
\end{equation}
On the other hand, the following relation holds as $j=1,2$:
\[
E(\lambda_{j})=\left(\beta_{0}+i \lambda_{j} \beta_{1}+\beta_{2} \lambda_{j}^{2}\right)e'(0,\lambda_{j})-\left(\alpha_{0}+i \lambda_{j} \alpha_{1}+\alpha_{2} \lambda_{j}^{2}\right)e(0,\lambda_{j})=0.
\]
Hence, we have
\[
e'(0, \lambda_{j})=-\frac{\left(\alpha_{0}+i \lambda_{j} \alpha_{1}+\alpha_{2} \lambda_{j}^{2}\right)e(0, \lambda_{j})}{\left(\beta_{0}+i \lambda_{j} \beta_{1}+\beta_{2} \lambda_{j}^{2}\right)} \quad \left(j=1,2\right).
\]
This formula yields
\begin{eqnarray} \nonumber
W\left\{e(0,\lambda_{1}), \overline{e(0, \lambda_{2})}\right\}&=&e(0,\lambda_{1})\overline{e(0, \lambda_{2})} \\ \nonumber
& \cdot& \left[-\frac{\left(\alpha_{0}+i \lambda_{1} \alpha_{1}+\alpha_{2} \lambda_{1}^{2}\right)}{\left(\beta_{0}+i \lambda_{1} \beta_{1}+\beta_{2} \lambda_{1}^{2}\right)}+\frac{\left(\alpha_{0}+i \lambda_{2} \alpha_{1}+\alpha_{2} \lambda_{2}^{2}\right)}{\left(\beta_{0}+i \lambda_{2} \beta_{1}+\beta_{2} \lambda_{2}^{2}\right)}\right].
\end{eqnarray}
The choice $\lambda_{2}=\lambda_{1}$ and some calculations give us 
\[
W\left\{e(0,\lambda_{1}), \overline{e(0, \lambda_{1})}\right\}=\frac{i\left|e(0,\lambda_{1})\right|^{2}\left(\lambda_{1}+\overline{\lambda_{1}}\right)\left[\delta_{1}+2 Im \lambda_{1}\delta_{2}-\delta_{3}\left|\lambda_{1}\right|^{2}\right]}{\left|\beta_{0}+i \lambda_{1} \beta_{1}+\beta_{2} \lambda_{1}^{2}\right|^{2}}.
\]
Taking the last relation into account with (\ref{2.4}) we have
\[
\left(\lambda_{1}^{2}-\overline{\lambda_{1}}^{2}\right)\int^{+\infty}_{0}\left|e(x,\lambda_{1})\right|^{2}\rho(x)dx \quad \quad \quad \quad \quad \quad \quad \quad \quad \quad \quad \quad \quad \quad \quad \quad \quad
\]
\[\quad \quad \quad \quad \quad -\frac{\left(\lambda_{1}^{2}-\overline{\lambda_{1}}^{2}\right)\left|e(0,\lambda_{1})\right|^{2}\left[i\delta_{1}+2 Im \lambda_{1}\delta_{2}-\delta_{3}\left|\lambda_{1}\right|^{2}\right]}{\left|\beta_{0}+ i \beta_{1}\lambda_{1}+\beta_{2}\lambda_{1}^{2}\right|^{2}}=0
\]
and we reach $Re \lambda_{1}=0$. Therefore, the zeros of the function $E(\lambda)$ lie only on the imaginary axis. \\

Let us prove that there are only finitely many. This is obvious if $E(0)\neq0$, because under this assumption the set of zeros can not have limit points. To verify that the number of zeros of $E(\lambda)$ is finite in general case, we can give an estimate for the distance between the neighboring zeros of the function $E(\lambda)$ (see \cite{Marchenko}, page 186). Thus, the lemma is proved. 
\end{proof}

\begin{lemma}
The zeros of the function $E(\lambda)$ are all simple.
\end{lemma}
\begin{proof}
\ Differentiating the equation
\begin{equation} \label{2.5}
-e''(x,\lambda)+q(x)e(x,\lambda)=\lambda^{2} \rho(x) e(x,\lambda)
\end{equation}
with respect to $\lambda$, one can show that 
\begin{equation} \label{2.6}
-\dot{e}''(x,\lambda)+q(x)\dot{e}(x,\lambda)=\lambda^{2} \rho(x) \dot{e}(x,\lambda)+2 \lambda \rho(x) e(x, \lambda)
\end{equation}
with the over dot denoting the derivative with respect to $\lambda$.\\

Multiplying (\ref{2.5}) by $\dot{e}(x,\lambda)$ and (\ref{2.6}) by $e(x, \lambda)$ subtracting them side by side and integrating over $\left[0, +\infty\right)$ we obtain, in view of the definition of the function $E(\lambda)$, the following relation
\begin{eqnarray} \nonumber
2\lambda \int_{0}^{+\infty} \left|e(x, \lambda)\right|^{2} \rho(x) dx &+& \frac{\left|e(0,\lambda)\right|^{2}(i \alpha_{1}+2 \lambda \alpha_{2})}{\beta_{0}+i \lambda \beta_{1}+ \beta_{2} \lambda^{2}} \\ \nonumber
&+& \frac{e(0, \lambda) (i \beta_{1}+ 2 \lambda \beta_{2})}{\beta_{0}+ i \lambda \beta_{1} + \beta_{2} \lambda^{2}}\cdot \frac{\left[-\left( \alpha_{0} + i \alpha_{1} \lambda +\alpha_{2} \lambda^{2}\right) e(0, \lambda)\right]}{\beta_{0}+i \lambda \beta_{1} + \beta_{2} \lambda^{2}} \\ \nonumber
&=& \frac{\dot{E}(\lambda) e(0, \lambda)}{\beta_{0}+i \lambda \beta_{1} + \beta_{2} \lambda^{2}}.
\end{eqnarray}
If we take $\lambda=i \mu_{k}$ in the last relation we have 
\begin{eqnarray} \nonumber
2 \mu_{k} \int_{0}^{+ \infty} \left|e(x, i \mu_{k})\right|^{2} \rho(x) dx&+& \frac{\left|e(0, i \mu_{k})\right|^{2}(\alpha_{1}+2 \mu_{k} \alpha_{2})}{\beta_{0}-\beta_{1} \mu_{k}-\beta_{2} \mu_{k}^{2}} \\ \nonumber
&+& \frac{\left|e(0, i \mu_{k})\right|^{2}\left(-\alpha_{0}+ \alpha_{1} \mu_{k} + \alpha_{2} \mu_{k}^{2}\right)\left(\beta_{1}+ 2 \mu_{k} \beta_{2}\right)}{\left(\beta_{0}-\beta_{1} \mu_{k}-\beta_{2} \mu_{k}^{2}\right)^{2}} \\ \label{2.7}
&=& -i \frac{\dot{E}(i \mu_{k})e(0, i \mu_{k})}{\beta_{0}-\beta_{1} \mu_{k}-\beta_{2} \mu_{k}^{2}} \quad (k=\overline{1,n}).
\end{eqnarray}
Some basic operations yield us the equation below:
\begin{eqnarray} \nonumber
2 \mu_{k} \int_{0}^{+ \infty} \left|e(x, i \mu_{k})\right|^{2} \rho(x) dx&+&\left|e(0, i\mu_{k})\right|^{2}\frac{-\delta_{1}-2\mu_{k}\delta_{2}+\mu_{k}^{2}\delta_{3}}{\left(\beta_{0}-\beta_{1} \mu_{k}-\beta_{2} \mu_{k}^{2}\right)}^{2} \\ \nonumber
&=&-i \frac{\dot{E}(i \mu_{k})e(0, i \mu_{k})}{\beta_{0}-\beta_{1} \mu_{k}-\beta_{2} \mu_{k}^{2}} \quad (k=\overline{1,n}).
\end{eqnarray}
It is clear from (\ref{1.3}) that the left hand side of the last equation is positive. Thus $\dot{E}(i \mu_{k})\neq 0$, i.e. the zeros of the function $E(\lambda)$ are simple. 
\end{proof}

The numbers $m_{k}$ given by the formula 
\[
m_{k}^{-2}:=\int^{+ \infty}_{0} \left|e(x, i\lambda_{k})\right|^{2} \rho(x) dx-\frac{\delta_{1}+2\delta_{2}\lambda_{k}-\delta_{3}\lambda_{k}^{2}}{2 \lambda_{k}\left(\beta_{0}-\beta_{1}\lambda_{k}-\beta_{2}\lambda_{k}^{2}\right)^{2}}\quad (k=\overline{1,n})
\]
are called \textit{the normalizing numbers} of the boundary value problem (\ref{1.1}), (\ref{1.2}). \\

Let us set
\[
S_{0}(\lambda)=
\begin{cases}
\frac{\overline{e_{0}\left(0, \lambda\right)}}{e_{0}(0, \lambda)}=e^{-2 i \lambda a}\frac{1+\tau e^{-2 i \lambda \alpha a}}{e^{-2 i \lambda \alpha a} +\tau}, & \beta_{2}=0,\\
\frac{\overline{e'_{0}\left(0, \lambda\right)}}{e'_{0}(0, \lambda)}=-e^{-2 i \lambda a}\frac{1-\tau e^{-2 i \lambda \alpha a}}{e^{-2 i \lambda \alpha a} -\tau}, & \beta_{2} \neq 0,
\end{cases} 
\]
and $\tau=\left(\alpha-1\right)/\left(\alpha+1\right)$. It is easy to verify that $S_{0}(\lambda)-S(\lambda)$ tends to zero as $\left|\lambda\right|\rightarrow +\infty$. 

\section{Main Equation}

\quad The following theorem is devoted to the construction and introduction of the main equation of the inverse problem.

\begin{theorem}
The kernel $K(x,y)$ of the representation (\ref{1.4}) satisfy the integral equation 
\begin{equation} \label{3.1}
F(x,y)+K(x,y)+\int_{\mu^{+}(x)}^{+\infty} K(x,t) F_{0}(t+y)dt-\tau K(x,2a-y)=0,
\end{equation}
where 
\[
F(x,y)=F_{s}(x,y)+\sum_{k=1}^{n}m_{k}^{2}e_{0}(x,i \lambda_{k})e^{-\lambda_{k}y},
\]
\[
F_{0}(x)=F_{0 s}(x)+\sum_{k=1}^{n}m_{k}^{2} e^{-\lambda_{k} x}.
\]
\end{theorem}
\begin{proof}
\ To obtain (\ref{3.1}), we substitute (\ref{1.4}) in (\ref{2.1}) and get
\begin{eqnarray} \nonumber
\frac{2 i \lambda w(x, \lambda)}{E(\lambda)}&-&\overline{e_{0}(x,\lambda)}+S_{0}(\lambda)e_{0}(x,\lambda)=\int^{+\infty}_{\mu^{+}(x)}K(x,t) e^{- i \lambda t}dt\\ \nonumber
&+&\left[S_{0}(\lambda)-S(\lambda)\right]e_{0}(x,\lambda)+\int^{+\infty}_{\mu^{+}(x)}\left[S_{0}(\lambda)-S(\lambda)\right]K(x,t) e^{ i \lambda t}dt\\ \label{3.2}
&-&\int^{+\infty}_{\mu^{+}(x)}S_{0}(\lambda)K(x,t) e^{i \lambda t}dt.
\end{eqnarray}

Multiplying both sides of (\ref{3.2}) with $\frac{e^{i \lambda y}}{2 \pi}$, then integrating with respect to $\lambda$ over $\left(- \infty, +\infty\right)$ the right hand side of (\ref{3.2}) becomes
\begin{eqnarray} \nonumber
K(x,y)&+&\frac{1}{2 \pi} \int_{-\infty}^{+\infty}\left[S_{0}(\lambda)-S(\lambda)\right]e_{0}(x,\lambda)e^{i \lambda y} d\lambda \\ \nonumber
&+& \int_{\mu^{+}(x)}^{+ \infty} K(x,t)\left\{\frac{1}{2 \pi} \int_{- \infty}^{\infty}\left[S_{0}(\lambda)-S(\lambda)\right]e^{i \lambda \left(t+y\right)}d \lambda\right\}dt\\ \label{3.3}
&-&\int_{\mu^{+}(x)}^{+ \infty} K(x,t)\left\{\frac{1}{2 \pi} \int_{- \infty}^{\infty}S_{0}(\lambda)e^{i \lambda \left(t+y\right)}d \lambda\right\}dt.
\end{eqnarray}

By elementary transforms we obtain
\begin{eqnarray} \nonumber
S_{0}(\lambda)&=& e^{-2 i \lambda a} \frac{\left(1-\tau^{2}\right)e^{2i \lambda \alpha a}}{1+\tau e^{2i \lambda \alpha a}}+\tau e^{-2 i \lambda a} \\ \nonumber
&=&e^{- 2 i \lambda a (1- \alpha)}\left(1-\tau^{2}\right)\sum_{k=0}^{\infty}(-1)^{k} \tau^{k} e^{2 i \lambda a \alpha k}+\tau e^{-2 i \lambda a}
\end{eqnarray}
as $\beta_{2}=0$. 
If we take into consideration this relation and the last integral of (\ref{3.3}) we have 
\begin{eqnarray} \nonumber
\frac{1}{2 \pi} \int_{- \infty}^{+\infty}S_{0}(\lambda)e^{i \lambda \left(t+y\right)}d \lambda&=&\left(1-\tau^{2}\right)\sum_{k=0}^{\infty}(-1)^{k} \tau^{k} \delta\left(t+y-2a\left(1- \alpha\right)+2a \alpha k\right)\\ \nonumber
&+&\tau \delta\left(t+y-2a\right),
\end{eqnarray}
where $\delta$ is the Dirac-delta function. \\

Hence, (\ref{3.3}) can be written as 
\begin{eqnarray} \label{3.4}
K(x,y)&+& F_{s}(x,y)+ \int_{\mu^{+}(x)}^{+\infty} K(x,t)F_{0s}(t+y)dt \\ \nonumber
&-&(1-\tau^{2})\sum^{\infty}_{k=0}(-1)^{k}\tau^{k}K(x,2a(1-\alpha)-2 a \alpha k-y)-\tau K(x,2a-y),
\end{eqnarray}
where
\begin{equation}
F_{s}(x,y)=\frac{1}{2}\left(1+\frac{1}{\sqrt{\rho(x)}}\right)F_{0}\left(y+ \mu^{+}(x)\right)+\frac{1}{2}\left(1-\frac{1}{\sqrt{\rho(x)}}\right)F_{0}\left(y+ \mu^{-}(x)\right),
\end{equation}
\begin{equation}
F_{0s}(x)=\frac{1}{2 \pi} \int^{+ \infty}_{- \infty} \left[S_{0}(\lambda)- S(\lambda)\right]e^{-i \lambda x} d\lambda.
\end{equation} \label{3.5}
The sum in (\ref{3.4}) equals zero for $y>\mu^{+}(x)$ (see \cite{Mamedov1}). Therefore, (\ref{3.3}) takes the form 
\[
K(x,y)+F_{s}(x,y)+\int_{\mu^{+}(x)}^{+\infty}K(x,t)F_{0s}(t+y)dt-\tau K(x,2a-y).
\]

Multiplying both sides of (\ref{3.2}) with $\frac{e^{i \lambda y}}{2 \pi}$, integrating with respect to $\lambda$ over $\left(- \infty, +\infty\right)$ and then using Jordan's lemma and the Residue theorem, on the left-hand side of (\ref{3.2}) we find
\begin{equation} \label{3.7}
\sum^{n}_{k=1}\frac{2 \lambda_{k}(\beta_{0}-\beta_{1}\lambda_{k}-\beta_{2}\lambda^{2}_{k})e(x, i\lambda_{k})e^{- \lambda_{k}y}}{i \dot{E}(i \lambda_{k})e(0, i\lambda_{k})}.
\end{equation}

Taking (\ref{2.7}) into account we can transform (\ref{3.7}) to the form
\begin{equation} \label{3.8}
-\sum_{k=1}^{n}m_{k}^{2}e(x,i\lambda_{k})e^{- \lambda_{k} y}.
\end{equation} 
From (\ref{1.4}) and (\ref{3.2}), we derive the equation
\begin{eqnarray} \nonumber
&-&\sum_{k=1}^{n}m_{k}^{2}\left[e_{0}(x,i\lambda_{k})e^{- \lambda_{k}y}+\int_{\mu^{+}(x)}^{+\infty}K(x,t)e^{- \lambda_{k}(t+y)}dt\right]=K(x,y) \\ \nonumber
&+&F_{s}(x,y)+\int_{\mu^{+}(x)}^{+\infty}K(x,t)F_{0s}(t+y)dt-\tau K(x,2a-y).
\end{eqnarray}
If we take (3.5) and (3.6) into consideration with (\ref{3.8}), we finally obtain (\ref{3.1}) for $y> \mu^{+}(x)$. The theorem is proved.
\end{proof}

Equation (\ref{3.1}) is called \textit{the main equation} of the inverse scattering problem of (\ref{1.1}), (\ref{1.2}). To form the main equation, it suffices to know the functions $F_{0}(x)$ and $F(x,y)$. In turn, to find the functions $F_{0}(x)$ and $F(x,y)$, it suffices to know the set of values
\[ 
\left\{S(\lambda) (-\infty <\lambda <+ \infty); \lambda_{k}; m_{k} (k=\overline{1,n})\right\}
\]
which is called \textit{the scattering data} of the boundary value problem (\ref{1.1}), (\ref{1.2}). With the given scattering data we can construct the functions $F_{0}(x)$, $F(x,y)$ and write out the main equation for the unknown function $K(x,y)$. Solving this equation we find the kernel $K(x,y)$ and with the help of (\ref{1.5}), (\ref{1.6}) we find the potential $q(x)$.
\begin{theorem}
For each fixed $x>0$, the main equation (\ref{3.1}) has a unique solution $K(x,y) \in L_{1}(0, +\infty)$.
\end{theorem}
\begin{proof}
\ The transition functions $F(x,y)$ and $F_{0}(x)$ have the similar properties to those of the transition functions for the problem with the spectral parameter linearly contained in the boundary conditions, thus the proof can be done analogously to \cite{Mamedov1}.
\end{proof}
With the help of the above theorem, we have:
\begin{corollary}
The potential $q(x)$ in problem (\ref{1.1}), (\ref{1.2}) is uniquely defined by the scattering data, i. e. if the scattering data of two problems with potentials $q(x)$ and $\tilde{q}(x)$ coincide, then $q(x)=\tilde{q}(x)$ a.e. on the half line. 
\end{corollary}

\begin{acknowledgement}
This work is supported by The Scientific and Technological Research Council
of Turkey (T\"{U}B\.{I}TAK).
\end{acknowledgement}

%\bibliography{mybibfile}

\end{document}